\documentclass{amsart}
\usepackage{amssymb,enumerate,amscd,amsthm,amsmath,graphics}
 \newtheorem{theorem}{Theorem}[section]
 \newtheorem{corollary}[theorem]{Corollary}
 \newtheorem{lemma}[theorem]{Lemma}
 \newtheorem{proposition}[theorem]{Proposition}
 \theoremstyle{definition}
 
 \theoremstyle{remark}

 \numberwithin{equation}{section}

\begin{document}

\title[On the tensor degree of finite groups]
 {On the tensor degree of finite groups}

\author[P. Niroomand]{Peyman Niroomand}
\address{School of Mathematics and Computer Science\\
Damghan University of Basic Sciences\\
Damghan, Iran}
\email{p$\_$niroomand@yahoo.com}

\author[F.G. Russo]{Francesco G. Russo}
\address{DIEETCAM\\
Universit\'a Degli Studi di Palermo\\
 Viale Delle Scienze, Edificio 8, 90128, Palermo, Italy
} \email{francescog.russo@yahoo.com}





\begin{abstract}
We study the number of elements  $x$ and $y$ of a finite group $G$ such that $x \otimes y= 1_{_{G \otimes G}}$ in the nonabelian tensor square $G \otimes G$ of $G$. This number, divided by  $|G|^2$, is called the tensor degree of $G$ and has connection with the exterior degree, introduced few years ago in [P. Niroomand and R. Rezaei, On the exterior degree of finite groups, Comm. Algebra 39 (2011), 335--343]. The analysis of  upper and lower bounds of the tensor degree allows us to find interesting structural restrictions for the whole group.
\end{abstract}

\maketitle

\section{Commutativity, exterior  and tensor degrees}

In the present paper all the groups are supposed to be finite. Brown and others \cite{bjr} wrote an influential contribution on a generalization of the usual abelian tensor product of  abelian groups. Following their terminology, the \textit{nonabelian tensor product} $G\otimes H$ of two groups $G$ and $H$ is the group generated by the symbols $g\otimes h$ with
defining relations
$xy\otimes h=(y^x\otimes h^x) (x\otimes h)$ and $x\otimes
hk=(x\otimes h) (x^h\otimes k^h)$ for all $x,y \in G$ and $h,k \in H$ (as usual, $y^x=x^{-1}yx$). In case $G=H$ and all actions are by conjugation, $G\otimes G$ is called the  \textit{nonabelian tensor square} of $G$.  \cite[Propositions 1,2,3]{bjr} describe the main calculus rules in $G \otimes H$, which are:
\begin{equation}\label{rules}
 (x^{-1}\otimes h)^x=(x\otimes h)^{-1}=(x\otimes h^{-1})^h;
\ \ \ \  (y\otimes k)^{xh} (x\otimes h)=(x\otimes
h) (y\otimes k)^{hx};
\end{equation}
\[
 y\otimes (h^xh^{-1})= (x\otimes h)^y (x\otimes h)^{-1};
\ \ \ \
 (x(x^{-1})^h)\otimes y=(x\otimes h) ((x\otimes
h)^{-1})^k;
\]
\[
 (y\otimes k)~^{(x\otimes h)}=(y\otimes
k)^{[x,h]};
\ \ \ \
 [x\otimes h,y\otimes k]=
(x(x^{-1})^h)(k^yk^{-1}).
\]
They allow us to conclude that
\[\kappa : x\otimes y \in G\otimes G \mapsto \kappa(x\otimes y)=[x,y] \in G'\]
is an epimorphism of groups such that $\ker \kappa =J_2(G)$ is a central subgroup of $G \otimes G$. Furthermore, $ \nabla(G)=\langle x \otimes x \ | \ x \in G \rangle \subseteq J_2(G) $ and $J_2(G)$ is important from the point of view of the algebraic topology (see \cite{bjr}). Defining  \textit{the nonabelian exterior square} $G \wedge G= (G \otimes G)/ \nabla(G)=\langle (x \otimes y) \nabla(G) \ | \ x,y \in G \rangle=\langle x \wedge y \ | \ x,y \in G \rangle$  of $G$, we can see easily that   \[\kappa' : x\wedge y \in G\wedge G \mapsto \kappa'(x\wedge y)=[x,y] \in G'\] is an epimorphism of groups such that $\ker \kappa' = M(G)=H_2(G,\mathbb{Z})$ is the \textit{Schur multiplier} of $G$, that is, the second integral homology group of $G$. The following diagram involves the third integral homology group of $G$ and the Whitehead functor $\Gamma$. It has exact rows and central extensions as columns (see \cite{bjr} for details).
\begin{equation} \label{diagram}\begin{CD}
@. @.  0 @. 0\\
@. @. @VVV   @VVV\\
H_3(G)@>>>\Gamma(G/G') @>>>J_2(G)@>>>H_2(G)@>>>0\\
@| @| @VVV @VVV\\
H_3(G)@>>>\Gamma(G/G') @>>>G\otimes G @>>>G\wedge G@>>>1 @. \\
@. @. @V\kappa VV   @V\kappa'VV\\
@. @. G' @= G'\\
@. @. @VVV   @VVV\\
@. @. 1 @. 1\\
\end{CD}\end{equation}
The \textit{exterior centralizer} of $x \in G$ is the set
\[C_G^\wedge(x)=\{a\in G  \ | \ a \wedge x=1_{_{G \wedge G}}\},\] which turns out to be a subgroup of $G$ (see \cite{pf}), and the \textit{exterior center} of $G$ is the set \[Z^\wedge(G)=\{g \in G \ | \ 1_{_{G \wedge G}}=g \wedge y \in G \wedge G, \forall y \in G\}={\underset{x \in G}\bigcap} C_G^\wedge(x)\] which  is a subgroup of the  center $Z(G)$ of $G$ (see \cite{pf}). We mention that the interest in studying $C_G^\wedge(x)$ and $Z^\wedge(G)$ is due to the fact that they allow us to decide whether $G$ is a \textit{capable group} or not, that is, whether $G$ is isomorphic to $E/Z(E)$ for some group $E$ or not. In analogy, we may consider the \textit{tensor centralizer}  \[C_G^\otimes(x)=\{a\in G  \ | \ a \otimes x=1_{_{G \otimes G}}\},\] of $x \in G$ and it turns out to be a subgroup of $G$ (see \cite[Theorem 3.1]{bk}), and the \textit{tensor center} \[Z^\otimes(G)=\{g \in G \ | \ 1_{_{G \otimes G}}=g \otimes y \in G \otimes G, \forall y \in G\}={\underset{x \in G}\bigcap} C_G^\otimes(x),\] which  is  a subgroup of $Z(G)$ (see \cite[Corollary 3.3]{bk}).

A combinatorial approach, in order to measure how far we are from $Z^\wedge(G)$, has been investigated in \cite{pr}, introducing the \textit{exterior degree} of $G$ \[d^\wedge(G)=\frac{|\{(x,y)\in G \times G \ | \  x \wedge y =1_{_{G \wedge G}}\}|}{|G|^2} =\frac{1}{|G|} \sum^{k(G)}_{i=1} \frac{|C^\wedge_G(x_i)|}{|C_G(x_i)|},\] where $k(G)$ is the number of conjugacy classes of $G$ and the second equality is  \cite[Lemma 2.2]{pr}. It is in fact clear that $d^\wedge(G)=1$ if and only if $G=Z^\wedge(G)$ so that the exterior degree represents the probability that two random chosen elements commute with respect to the operator $\wedge$. On the other hand, $d^\wedge(G)$ has been connected with the \textit{commutativity degree}  \[d(G)=\frac{|\{(x,y)\in G \times G \ | \  [x,y] =1\}|}{|G|^2}=\frac{1}{|G|^2} \sum^{k(G)}_{i=1} |C_G(x_i)|\] of $G$, studied by Gustafson and others \cite{elr, er, gr, l1, l2, rusin}. The reader may refer also to \cite{new,prf} for recent studies on the topic. Of course, $d(G)=1$ if and only if $G$ is abelian. On the other hand,  among groups with exterior degree equal to 1,  we find all cyclic groups, but not necessarily  all abelian groups (see \cite{bacon,bk,pr}). Then, roughly speaking, $d^\wedge(G)$ gives a measure of how far is $G$ from being cyclic. It is interesting to note that  \[d^\wedge(G) \le d(G),\]
by \cite[Theorem 2.3]{pr}. Here we introduce the \textit{tensor degree}
\[d^\otimes(G) = \frac{|\{(x,y)\in G \times G \ | \  x \otimes y =1_{_{G \otimes G}}\}|}{|G|^2} ,\] evaluating the distance of $G$ from being equal to $Z^\otimes(G)$, since  $d^\otimes(G)=1$ if and only if $Z^\otimes(G)=G$. On the other hand, one may easily check that $d^\otimes (G)=1$ if and only if $G$ is trivial, by using \cite[Proposition 18]{ellis}.

We note that there are several examples of abelian groups with exterior degree different from 1 and here we show,  not only examples of abelian groups with tensor degree different from 1, but also examples of abelian groups with different values of tensor and exterior degrees. In fact, the study of $d^\otimes(G)$ will present a new perspective  of investigation of  specific subclasses of nilpotent groups. We will prove restrictions on $d^\otimes(G)$ of numerical nature, which will influence the structure of $G$, and relations among $d^\otimes(G)$, $d^\wedge(G)$ and $d(G)$.

\section{Fundamental inequalities for the tensor degree}
 In a group $G$,  $C^\wedge_G(x)$ and $C^\otimes_G(x)$ are  normal subgroups of  $C_G(x)$ (see \cite{bk,pf}).

\begin{lemma}\label{section} Let $x$ be an element of a group $G$. Then
\[|C_G(x):C^\otimes_G(x)| \le |J_2(G)|.\]
Furthermore, $|C_G(x):C^\otimes_G(x)| \le |M(G)| \ |\nabla(G)|.$
\end{lemma}

\begin{proof} The map   \[\varphi: yC^\otimes_G(x) \in C_G(x)/C^\otimes_G(x) \longmapsto y \otimes x \in J_2(G)\]
satisfies the condition $\varphi (abC^\otimes_G(x))=ab \otimes x = (a \otimes x)^b \ (b \otimes x)=(a \otimes x) \ (b \otimes x)=\varphi(aC^\otimes_G(x)) \ \varphi(bC^\otimes_G(x))$ for all $a,b \in C_G(x)$, where \eqref{rules} have been applied. Furthermore,  $\ker \varphi = \{yC^\otimes_G(x) \ | \ y \otimes x = 1_{G \otimes G}\}=C^\otimes_G(x)$. Then $\varphi$ is a monomorphism and  $|C_G(x):C^\otimes_G(x)| \le |J_2(G)|$.  On the other hand, we know from \cite{bjr} that the map $\pi : a \otimes b \in J_2(G) \mapsto  a \wedge b \in M(G)$ is an epimorphism of groups such that $J_2(G)/\ker \pi=J_2(G)/\nabla(G) \simeq M(G)$. Thus
$|C_G(x):C^\otimes_G(x)| \le |M(G)| \ |\nabla (G)|$.
\end{proof}

The above bound shows that the section $C_G(x)/C^\otimes_G (x)$  depends on the size of $J_2(G)$, or, equivalently, from that of $\nabla(G)$ and $M(G)$. This agrees with the homological sequences of \eqref{diagram}. The following lemma deals with different aspects and correlates the tensor centralizers with the tensor degree.

\begin{lemma} \label{nice}
Let $x_1, \ldots, x_{k(G)}$ be  a system of representatives
 for the conjugacy classes of the group $G$. Then
\[d^\otimes(G)=  \frac{1}{|G|} \sum^{k(G)}_{i=1} \frac{|C^\otimes_G(x_i)|}{|C_G(x_i)|}.\]
\end{lemma}

\begin{proof} The idea of \cite[Proof of Lemma 2.2]{pr} may be adapted here. Let $C_1,\ldots, C_{k(G)}$ be the conjugacy classes of $G$ and $x_i \in C_i$ for $i=1,2, \ldots, k(G)$. For every $y \in C_i$ there exists a $g \in G$ such that $y=x^g_i$. This implies $|C^\otimes_G(y)|=|C^\otimes_G(x_i)|$, hence
\[|G|^2 d^\otimes(G)=\sum_{x \in G}|C^\otimes_G(x)|=\sum^{k(G)}_{i=1} \sum_{x \in C_i}|C^\otimes_G(x_i)|=\sum^{k(G)}_{i=1}|G:C_G(x_i)||C^\otimes_G(x_i)|\]
\[=|G| \sum^{k(G)}_{i=1}\left| \frac{C^\otimes_G(x_i)}{C_G(x_i)}\right|.\]
\end{proof}

We may compare commutativity  and  tensor degrees in the following way.

\begin{theorem}\label{fundamental1} Let $G$ be a group and $p$ be the smallest prime divisor of $|G|$. Then
\[\frac{d(G)}{|J_2(G)|}+\frac{|Z^\otimes(G)|}{|G|} \left(1-\frac{1}{|J_2(G)|}\right) \le d^\otimes(G)
\le d(G)-\left(1-\frac{1}{p}\right)\left(\frac{|Z(G)|-|Z^\otimes(G)|}{|G|}\right).\]
\end{theorem}

\begin{proof}  The idea of \cite[Proof of Theorem 2.3]{pr} may be applied, thanks to the preliminaries which we have done. Let $x \not \in Z^\otimes(G)$.  From Lemma \ref{section}, $|C^\otimes_G(x)|/|C_G(x)| \ge 1/|J_2(G)|$. From Lemma \ref{nice} and the equality $d(G)=\frac{k(G)}{|G|}$, we deduce
\[d^\otimes(G)=\frac{1}{|G|} \sum^{k(G)}_{i=1}\left| \frac{C^\otimes_G(x_i)}{C_G(x_i)}\right|\ge \frac{1}{|G|} \left( |Z^\otimes(G)|+\frac{k(G)-|Z^\otimes(G)|}{|J_2(G)|}\right)\]
\[=\frac{k(G)}{|G| \ |J_2(G)|}+\frac{|Z^\otimes(G)|}{|G|}\left(1-\frac{1}{ |J_2(G)|}\right)=\frac{d(G)}{|J_2(G)|}+\frac{|Z^\otimes(G)|}{|G|} \left(1-\frac{1}{|J_2(G)|}\right).\]

Conversely, $|G:C^\otimes_G(x)| \ge p$ implies that
\[d^\otimes(G)=\frac{1}{|G|} \sum^{k(G)}_{i=1}\left| \frac{C^\otimes_G(x_i)}{C_G(x_i)}\right|\le \frac{|Z^\otimes(G)|}{|G|}+ \frac{1}{p} \left(\frac{|Z(G)|-|Z^\otimes(G)|}{|G|}\right)+\frac{k(G)-|Z(G)|}{|G|}\]
\[=d(G)-\frac{p-1}{p}\left(\frac{|Z(G)|-|Z^\otimes(G)|}{|G|}\right).\]
\end{proof}

Theorem \ref{fundamental1} has analogies with \cite[Theorem 2.3]{pr}.  On the other hand, the literature on $M(G)$ is richer than the literature on $J_2(G)$ and it may be useful to rewrite the lower bound of Theorem \ref{fundamental1} in the following way.

\begin{corollary} Let $G$ be a group. Then
\[\frac{d(G)}{|M(G)| |\nabla(G)|}+\frac{|Z^\otimes(G)|}{|G|} \left(1-\frac{1}{|M(G)| |\nabla(G)|}\right) \le d^\otimes(G).\]
\end{corollary}

From \cite[Proposition 18]{ellis}, an abelian group $G$ has trivial $Z^\otimes (G)$.
This means that the term $\frac{|Z^\otimes(G)|}{|G|} \left(1-\frac{1}{|J_2(G)|}\right)$ vanishes in the lower bound of Theorem \ref{fundamental1}; $J_2(G)=G\otimes G$;
$d(G)=1$; the term $\frac{|Z(G)|-|Z^\otimes(G)|}{|G|}$ becomes equal to $\frac{|Z(G)|-1}{|G|}$ . Hence we have
another interesting consequence of Theorem \ref{fundamental1}.

\begin{corollary}
An abelian group $G$ satisfies
$\frac{1}{|G|}+\frac{|G|-1}{|G||G\otimes G|}\leq d^\otimes(G)\leq \frac{1}{p}+ \frac{p-1}{p|G|}$.
Moreover, if $|G|\rightarrow \infty$, then $0\le d^\otimes(G)\leq \frac{1}{p}$.
\end{corollary}

The well--known notion of  \textit{Schur cover}, which can be found in \cite{tappe, bjr}, gives a condition of equality among the tensor degree and the commutativity degree. 

\begin{proposition}Let $G$ be a perfect group and $G^{*}$ be a Schur cover of $G$. Then
$d^{\otimes}(G)=d(G^{*})$. 
\end{proposition}
\begin{proof}By the definition of Schur cover, there is an exact sequence \[1\rightarrow A \rightarrow G^{*}\stackrel{\pi}\rightarrow  G\rightarrow 1,\] in which $A\cong M(G)$. From \cite[Proposition 7]{bjr}, we have a homomorphism $\xi:G\otimes G\rightarrow G^{*}$ given by $\xi(g\otimes h)=[g_1,h_1]$ such that $\pi(g_1)=g, \pi(h_1)=h$.
Since $G\otimes G \simeq G^{*}$ by using \cite[Corollary 1]{bjr}, we have \[|\{(g,h)\in G\times G~|~g\otimes h=1_{G\otimes G}\}||M(G)|=|\{(g_1,h_1)\in G^{*}\times G^{*}~|~[g_1,h_1]=1\}|,\] as required.
\end{proof}

Another fundamental relation among $d(G)$, $d^\wedge(G)$ and $d^\otimes(G)$ is listed below.

\begin{theorem}\label{fundamental2}
Let $G$ be a group. Then   $d^\otimes(G) \le d^\wedge(G) \le d(G)$.
\end{theorem}

\begin{proof}
$d^\wedge(G) \le d(G)$ follows from \cite[Theorem 2.3]{pr}. Let $x \not \in Z^\otimes(G)$. We have $g \in C^\wedge_G(x)$ if and only if $g \wedge x = 1_{G \wedge G}$ if and only if $(g \otimes x) \nabla(G)=\nabla(G)$ if and only if $g \otimes x \in \nabla(G)$. This condition is weaker than the condition $g \otimes x = 1_{G \otimes G}$, characterizing the elements of
$ C^\otimes_G(x)$. Then $ C^\otimes_G(x)\subseteq C^\wedge_G(x) \subseteq C_G(x)$. From this fact,  Lemma \ref{nice} and \cite[Lemma 2.2]{pr}, we conclude
\[d^\otimes(G) = \frac{1}{|G|} \sum^{k(G)}_{i=1}\left| \frac{C^\otimes_G(x_i)}{C_G(x_i)} \right| \le \frac{1}{|G|} \sum^{k(G)}_{i=1}\left| \frac{C^\wedge_G(x_i)}{C_G(x_i)} \right| =d^\wedge(G).\]
\end{proof}

From \cite[Theorem 2.3]{pr} and Theorem \ref{fundamental2}, $G$ is \textit{unidegree}, or \textit{right unidegree}, if the upper bound $d^\wedge(G) = d(G)$ is achieved. A group $G$ is called \textit{unicentral} if   $Z^\wedge(G)=Z(G)$. Unicentral groups have been studied in \cite{tappe} and \cite[Corollary 2.5]{pr} shows that right unidegree groups are unicentral.
Now we say that $G$ is \textit{left unidegree} if the lower bound $d^\otimes(G) = d(G)$ is achieved. Finally, $G$ is said to be \textit{left and right unidegree} if it is both left and right unidegree. By the concept of left unidegree introduced here, this condition and
Theorem 2.7 imply that of right unidegree. A complete classification of groups, whose center is equal to its tensor center, is not available, to the best of our knowledge, against the classification of unicentral groups in \cite{tappe}.

\begin{corollary} If $G$ is left unidegree, then $Z(G)=Z^\otimes(G)$. Furthermore, if $G$ is left and right unidegree, then $G$ is unicentral and $Z(G)=Z^\otimes(G)$.
\end{corollary}

\begin{proof} If $d^\otimes(G)=d(G)$, then $Z(G)=Z^\otimes(G)$ by the upper bound in Theorem \ref{fundamental1}. Furthemore, if $d^\wedge(G)=d(G)$, then we apply \cite[Corollary 2.5]{pr} and the result follows.
\end{proof}

\section{Sharpening upper and lower bounds for the tensor degree}

The present section is devoted to improve the numerical restrictions on the tensor degree and to find relations among quotients and subgroups. \cite[Proposition 2.6]{pr} has analogies in the context of the tensor degree.

\begin{proposition}\label{quotient}If $N$ is a normal subgroup of a group $G$, then
$d^\otimes(G) \le d^\otimes(G/N).$ The equality holds, if $N\subseteq Z^\otimes(G)$.
\end{proposition}

\begin{proof}
\[|G|^2 \ d^\otimes (G)=  \sum_{x \in G} |C^\otimes_G(x)| = \sum_{xN\in G/N} \sum_{n \in N}|C^\otimes_G(xn)|\]
\[=\sum_{xN\in G/N} \sum_{n \in N} \frac{|C^\otimes_G(xn)N|}{|N|} \ |C^\otimes_G(xn)\cap N|\le \sum_{xN\in G/N} \sum_{n \in N}|C^\otimes_{G/N}(xN)| \ |C^\otimes_G(xn)\cap N|\]
\[=\sum_{xN\in G/N} |C^\otimes_{G/N}(xN)| \ \sum_{n\in N}|C^\otimes_G(xn)\cap N|\]
\[\le |N|^2 \sum_{xN \in G/N}|C^\otimes_{G/N}(xN)|  = |G|^2 \ d^\otimes(G/N).\]
For each central subgroup $N$ of $G$, \cite[Proposition 9]{bjr} ensures the exactness of the sequence $(G \otimes N) \times (N \otimes G)  {\overset{\alpha} \longrightarrow} G \otimes G {\overset{\beta} \longrightarrow} (G/N)\otimes (G/N) \longrightarrow 1$ for suitable homomorphisms $\alpha$ and $\beta$. This is  our case. Furthermore, if $N\subseteq Z^\otimes(G)$, then  $\mathrm{Im} \ \alpha=1_{_{G \otimes G}}$ and  $G/N \otimes G/N\simeq G \otimes G$ so that $d^\otimes(G) = d^\otimes(G/N)$.
\end{proof}

The next corollary explains better the conditions of equality in Proposition \ref{quotient}.

\begin{corollary} Let $N$ be a normal subgroup of a group $G$. Then
the following statements are equivalent:
\begin{itemize}
\item[(i)] $N\subseteq Z^\otimes(G);$
\item[(ii)]$d^\otimes(G) = d^\otimes(G/N);$
\item[(iii)]$G\otimes G\cong G/N\otimes G/N.$
\end{itemize}
\end{corollary}
\begin{proof} (i) implies (ii) by Proposition \ref{quotient}. Conversely,  $d^\otimes(G) = d^\otimes(G/N)$ implies that $N\subseteq C^{\otimes}_G(x)$ for all $x\in G$, then $N\subseteq Z^\otimes(G).$ Finally, the equivalency of (ii) and (iii) are obtained directly from \cite[Proposition 16]{ellis}.
\end{proof}

The next result is going to improve the upper bound in Theorem \ref{fundamental1}.

\begin{proposition}\label{bestupperbound} Let $G$ be a group, $p$ the smallest prime divisor of $|G|$ and $x \not \in Z(G)$ such that $C^\otimes_G(x)\not=C_G(x)$. Then
\[d^\otimes(G)\le d(G)-\left(1-\frac{1}{p}\right)\left(\frac{|Z(G)|-|Z^\otimes(G)|-1}{|G|}\right).\]
In particular, if $Z^\otimes(G)=1$, then
$d^\otimes(G)\le d(G)-\left(1-\frac{1}{p}\right) \frac{|Z(G)|}{|G|}.$
\end{proposition}

\begin{proof} We may apply the same argument, which has been used to prove the upper bound of Theorem \ref{fundamental1}. The first  inequality may be specified as
\[d^\otimes(G)\le \frac{|Z^\otimes(G)|}{|G|}+ \frac{1}{p} \left(\frac{|Z(G)|-|Z^\otimes(G)|}{|G|}\right)+\frac{1}{p |G|}+\frac{k(G)-|Z(G)|-1}{|G|}\]
\[=\frac{1}{|G|} \left(\left(1-\frac{1}{p}\right)|Z^\otimes(G)|+  \left(\frac{1}{p}-1\right) |Z(G)|+\frac{1}{p}+k(G)-1\right) \]
\[= d(G)-\left(1-\frac{1}{p}\right)\left(\frac{|Z(G)|-|Z^\otimes(G)|-1}{|G|}\right).\]
The rest is clear and the result follows.
\end{proof}

The extremal case of $Z^\otimes(G)=1$ is described by the next result and has analogies with \cite[Theorem 2.8]{pr}. There are also analogies for the commutativity degree in \cite{elr,gr, l1,l2,rusin}.

\begin{theorem}\label{extreme} Let $G$ be a nonabelian group with $Z^\otimes(G)=1$ and $p$ be the smallest prime dividing $|G|$. Then $d^\otimes(G) \le \frac{1}{p}.$
\end{theorem}

\begin{proof}
First we claim that, if $Z(G)\cap G'\not=1$, then there exists an $x \not \in Z(G)$ such that  $C^\otimes_G(x)\not=C_G(x)$.
Assume that  $C^\otimes_G(x)=C_G(x)$ for all $x \not \in Z(G)$. Since $Z(G) \subseteq {\underset{x \not\in Z(G)}\bigcap} C^\otimes_G(x)$ and $G' \subseteq {\underset{x \in Z(G)}\bigcap} C_G^\otimes(x)$, we have $G' \cap Z(G) \subseteq Z^\otimes(G)$. This implies
$Z(G)\cap G'=1$, which is a contradiction. The first claim follows.

The second claim is  that $d(G) \le \frac{1}{p}+\left(1-\frac{1}{p} \right) \frac{|Z(G)|}{|G|}$.
Arguing as in Theorem \ref{fundamental1}, we find
\[d(G)=\sum_{x \in G} \frac{|C_G(x)|}{|G|}\le \frac{1}{p} \left(\frac{|Z(G)|}{|G|}+ \frac{|G|-|Z(G)|}{|G|} \right) =\frac{1}{p}+\left(1-\frac{1}{p} \right) \frac{|Z(G)|}{|G|}.\]

The third claim is that $d(G) \le \frac{1}{p}$, provided $G$ is nonabelian with $Z(G) \cap G'=1$. This can be found in  \cite[Proposition 2.7]{pr}.

 Now we may proceed to prove the result and  there is no loss of generality to assume $Z(G)\cap G'\not=1$, by the third claim and Theorem \ref{fundamental2}. Now the first and the second claims, combined with Proposition \ref{bestupperbound}, imply
\[d^\otimes(G) \le  d(G)-\left(1-\frac{1}{p}\right) \frac{|Z(G)|}{|G|} \le \frac{1}{p}+\left(1-\frac{1}{p} \right) \frac{|Z(G)|}{|G|}-\left(1-\frac{1}{p}\right) \frac{|Z(G)|}{|G|}=\frac{1}{p}.\]
\end{proof}

\section{Computations for elementary abelian and extraspecial $p$--groups}
The present section is devoted to compute the tensor degree for some classes of $p$--groups, widely studied in literature. We follow the phylosophy of \cite{p}, where the case of  the exterior degree has been studied. As usual,  $C_{p^n}$ denotes the cyclic group of order $p^n$. The tensor degree of  elementary abelian $p$-groups is described by the next result.

\begin{proposition}\label{elementaryabelian}
If $G \simeq C_p \times \ldots \times C_p=C^{(n)}_p$ is the direct product of $n\ge 1$ copies of $C_p$, then $d^\otimes(G)= \frac{2p^n-1}{p^{2n}}$.
\end{proposition}

\begin{proof}
We claim that $C^\otimes_G(x)=1$ for all nontrivial $x \in G$. In our case, $G \otimes G$ is the usual abelian tensor product of two abelian groups. Then
\[G \otimes G \simeq C^{(n)}_p \otimes C^{(n)}_p \simeq C^{(n^2)}_p.\] If we write $G = \langle a_1 \rangle \times \ldots \langle a_n\rangle$, then $G \otimes G = \prod^n_{i=1} \prod^n_{j=1} \langle a_i \otimes a_j\rangle$ so that $a_i \otimes a_j \not=1$ for all $i \not=j$. Since $a_i \otimes a_i\not=1$ for all $i=1,\ldots,n$, we have $C^\otimes_G(a_i)=1$ for all $i=1, \ldots, n$. Then
\[d^\otimes(G)=\frac{1}{|G|^2} \sum_{x \in G} |C^\otimes_{G}(x)|= \frac{1}{p^{2n}} \ \left( p^n + (p^n-1) \cdot 1\right)=\frac{2p^n-1}{p^{2n}}.\]
\end{proof}

The following $p$--groups have big sections which are elementary abelian and have interest in several areas of group theory.
Let \[E_1= \langle a,b,c \ | \ a^p=b^p=c^p=1, [a,c]=[b,c]=1, [a,b]=c\rangle\]  be the extra--special $p$--group of order $p^3$ and of exponent $p$,
\[D_{2^n}=C_{2^{n-1}} \rtimes C_2=\langle a, b \ | \ b^2=a^{2^{n-1}}=1, b^{-1}ab=a^{-1} \rangle \]
 be the dihedral 2--group of order $2^n$ and
\[Q_{2^n}= \langle a,b \ | \  b^2=a^{2^{n-2}},  b^{-1}ab=a^{-1} \rangle\]
be the generalized quaternion 2--group of order $2^n$. We note that $Q_{2^n}/Z(Q_{2^n}) \simeq D_{2^n}/Z(D_{2^n}) \simeq C^{(n-1)}_2$ for all $n \ge1 $.

\begin{lemma}\label{extraspecial} Let $H$ be an extraspecial $p$--group. Then $Z^\otimes(H)=Z(H)=H'$, when $H \not \simeq E_1,  Q_8, D_8$. Furthermore,   $Z^\otimes(E_1)=Z^\otimes(D_8)=Z^\otimes(Q_8)=1$.
\end{lemma}
\begin{proof} From GAP \cite{gap}, we have that $Z^\otimes(E_1)=Z^\otimes(D_8)=Z^\otimes(Q_8)=1$. On the other hand, $H$ is extraspecial, then $H'=Z(H) \simeq C_p$. Necessarily, we should have $Z^\otimes(H)=H'=Z(H)$,  when $H \not \simeq E_1,  Q_8, D_8$.
\end{proof}

We are going to calculate the tensor degree for the groups in Lemma \ref{extraspecial}.

\begin{theorem}\label{computations}
Let $p$ be a prime and $m \ge 1$.
\begin{itemize}
\item[(i)]$d^\otimes(Q_{2^n})=\frac{2^{n-3}+2^{n-4}+1}{2^n}$ for all $n>3$ and $d^\otimes(Q_8)=\frac{1}{4}$ for $n=3$.
\item[(ii)] $d^\otimes(D_{2^n})=\frac{2^{n-3}+2^{n-4}+1}{2^n}$ for all $n >3$ and  $d^\otimes(D_8)=\frac{5}{16}$ for $n=3$.
\item[(iii)]$d^\otimes(E_1)=\frac{2p^2+p-2}{p^5}$.
\item[(iv)]Let $H$ be an extra--special $p$--group  of $|H|=p^{2m+1}$  and not isomorphic with $E_1, Q_8, D_8$. Then
$d^\otimes(H)=\frac{2p^m-1}{p^{4m}}.$
\end{itemize}
\end{theorem}

 \begin{proof}
(i). We need to have in mind the presentation of $Q_{2^n}$ and some computations in \cite[Section 4]{bjr}. Let $a$ and $b$ be two generators of  $Q_{2^n}$ and $1 \not=a \otimes b \in Q_{2^n} \otimes Q_{2^n}$. For all $l\ge1$,   \cite[Equations 4.3 and 4.5]{bjr} imply
\[b\otimes a^l = (b\otimes a)^l \ (a\otimes a)^{l(l-1)}, \ \ \
\ a^l \otimes b=(a \otimes b)^l \ (a\otimes a)^{l(l-1)}\]
so that the condition \[b\otimes a^l = (b\otimes a)^l \ (a\otimes a)^{l(l-1)}=(a \otimes b)^l \ (a\otimes a)^{l(l-1)}=a^l \otimes b\]
is satisfied if and only if \[(b \otimes a)^l =(a \otimes b)^l \] for all $l \ge 1$. This cannot happen, because \cite[Equation at p.190, line +10]{bjr} implies $(a \otimes b)^l=1$ for all $l \ge1$, and, in particular, $a \otimes b=1$, which is a contradiction.
We conclude that $b \not \in C^\otimes_{Q_{2^n}}(a^l)$ for any choice of $l \ge1$. In a similar way, \cite[Equations 4.7 and 4.8]{bjr} imply
\[ba^k\otimes a^l = (b\otimes a)^l \ (a\otimes a)^{lk+l(l-1)}, \ \ \
\ a^l \otimes ba^k=(a \otimes b)^l \ (a\otimes a)^{lk+l(k-1)}\]
for all $l,k \ge1$ and we  find again $ba^k\otimes a^l=a^l \otimes ba^k$ if and only if $(b \otimes a)^l =(a \otimes b)^l$, which is a contradiction. Then
 $ba^k  \not \in C^\otimes_{Q_{2^n}}(a^l)$ for all $l,k \ge1$. We conclude that $C^\otimes_{Q_{2^n}}(a^l)$ is formed only by powers of $a$. Now an explicit computation (or looking at \cite[Equation 4.5]{bjr}) shows that only even powers of $a$ commute with $a^{2k+1}$ with respect to $\otimes$. Then $C^\otimes_{Q_{2^n}}(a^{2k+1})=\langle a^2\rangle$ and $C^\otimes_{Q_{2^n}}(a^{2k})= \langle a \rangle$. Furthermore, for all $i=0, 1, 2, \ldots, 2^{n-1}$ we may also deduce $C^\otimes_{Q_{2^n}}(a^ib)= 1$, after all we have said. Then
\[d^\otimes(Q_{2^n})=\frac{1}{|Q_{2^n}|^2} \sum_{x \in Q_{2^n}} |C^\otimes_{Q_{2^n}}(x)|\]
\[= \frac{1}{2^n} \ ( 2^n+ 2^{n-2}|\langle a^2 \rangle|+(2^{n-2}-1) |\langle a \rangle|+(2^{n-1}-1) \cdot 1 +1)\]
\[= \frac{1}{2^n} \ \left(\frac{ 2^n+ 2^{n-2} \cdot 2^{n-2} +(2^{n-2}-1) \cdot 2^{n-1}+2^{n-1}}{2^n}\right)\]
\[= \frac{4+2^{n-2}+2(2^{n-2}-1)+2}{2^{n+2}}=\frac{2^{n-3}+2^{n-4}+1}{2^n}.\]
The case  $n=3$, that is, $Q_8$ can be solved by using GAP \cite{gap} and shows $d^\otimes(Q_8)=\frac{1}{4}$.

(ii). Exactly the same argument of (i) may be applied for $D_{2^n}$ when $n>3$ and we find
$d^\otimes(D_{2^n})=d^\otimes(Q_{2^n})$ (see \cite[Section 4]{bjr}).  Similarly, GAP \cite{gap} allows us to conclude
$d^\otimes(D_8)=\frac{5}{16}$ in case $n=3$.

(iii). From GAP \cite{gap}, or by a direct computation, one can see that $E_1 \otimes E_1 \simeq C^{(6)}_p$. Thus \cite[Proposition 3.5]{bacon}  implies
\[E_1 \otimes E_1 = \langle a \otimes   a\rangle \times \langle a \otimes   b\rangle \times \langle a \otimes   c\rangle
\times \langle b \otimes   c\rangle \times \langle b \otimes   a\rangle \times \langle b \otimes   b\rangle.\]
Applying \cite[Proposition 3.5]{bacon} again, we get for all $1 \le i \le p$ that $C^\otimes_{E_1}(c^i)=c^i$ and for all $x \in E_1- \{1, c, c^2, \ldots, c^{p-1}\}$ that $C^\otimes_{E_1}(x)=1$. Then
\[d^\otimes(E_1)=\frac{1}{|E_1|^2} \sum_{x \in E_1} |C^\otimes_{E_1}(x)|\]
\[= \frac{1}{p^6} \left( p^3+ {\underbrace{|C^\otimes_{E_1}(c)|+ |C^\otimes_{E_1}(c^2)| + \ldots + |C^\otimes_{E_1}(c^{p-1})|}_{(p-1)- \mathrm{times}}} +(p^3-p)\right)\]
\[=\frac{1}{p^6} (p^3+(p-1)p+p^3-p)= \frac{2p^2+p-2}{p^5}.\]

(iv). From Proposition \ref{quotient}, we note that the tensor degree of a group is the same if we factorize through its tensor center.  This and Lemma \ref{extraspecial} imply $d^\otimes (H)= d^\otimes (H/Z^\otimes(H))=d^\otimes(H/H')$.  Since $H/H' \simeq C^{(2m)}_p$, \[d^\otimes(H)=d^\otimes(H/H')=\frac{2p^m-1}{p^{4m}}.\]
 \end{proof}

It is instructive to compare the  results of the present section with some of \cite{elr,gr,l1,l2,pr,p}. We will confirm not only Theorems \ref{fundamental1}, \ref{fundamental2} and \ref{computations}, but will verify that exterior degree, tensor degree and commutativity degree are  different group invariants. To conveniece of the reader, we have appended a list below.

\begin{corollary} The following inequalities are true:
\begin{itemize}
\item[(i)]$d^\otimes (E_1)= \frac{2p^2+p-2}{p^5}$; $d^\wedge(E_1)=\frac{p^3+p^2-1}{p^5}$; $d(E_1)=\frac{p^2+p-1}{p^3}$. In particular, $d^\otimes (E_1)<d^\wedge(E_1)<d(E_1)$ are  proper for all primes $p \ge2$.
\item[(ii)]$d^\otimes(Q_{2^n})=\frac{2^{n-3}+2^{n-4}+1}{2^n}$ and $\frac{2^{n-2}+3}{2^n} =d^\wedge(Q_{2^n})=d(Q_{2^n})$ for all $n > 3$. Moreover,  $d^\otimes(Q_8)=\frac{1}{4}$ and $\frac{5}{8}=d^\wedge(Q_8)=d(Q_8)$ in case $n=3$. In particular, $d^\otimes (Q_{2^n})<d^\wedge(Q_{2^n})=d(Q_{2^n})$ for all $n \ge3$.
\item[(iii)]$d^\otimes(D_{2^n})=\frac{2^{n-3}+2^{n-4}+1}{2^n}$ and $d^\wedge(D_{2^n})=d(D_{2^n})=\frac{2^{n-2}+3}{2^n}$ for all $n > 3$. Moreover, $d^\otimes(D_8)=\frac{5}{16}$ and  $d^\wedge(D_8)=d(D_8)=\frac{5}{8}$ in case $n=3$. In particular, $d^\otimes (D_{2^n})<d^\wedge(D_{2^n})=d(D_{2^n})$ for all $n \ge3$.
\item[(iv)]$d^\otimes(C^{(n)}_p)= \frac{2p^n-1}{p^{2n}}$;  $d^\wedge(C^{(n)}_p)=\frac{p^n+p^{n-1}-1}{p^{2n-1}}$; $d(C^{(n)}_p)=1 $. In particular,
$d^\otimes(C^{(n)}_p)<d^\wedge(C^{(n)}_p)<d(C^{(n)}_p)$ are proper for all $n \ge1$ and for all primes $p \ge2$.
\end{itemize}
\end{corollary}

\begin{proof}(i). See Theorem \ref{computations} (ii), \cite[Theorem 2.2]{p},
and \cite[Theorem A]{er}. (ii) and (iii). See Theorem \ref{computations} (i), \cite[Examples 3.1 and 3.2]{pr}, \cite[Remark 4.2]{l2}. (iv). See Proposition \ref{elementaryabelian} and \cite[Example 3.3]{pr}.
\end{proof}

\end{document}